\documentclass[reqno,10pt]{amsart}

\usepackage{amsmath}
\usepackage{amssymb}
\usepackage{amsthm}
\usepackage[alphabetic]{amsrefs}
\usepackage[foot]{amsaddr}
\usepackage{mathrsfs}
\usepackage{geometry}
\usepackage{graphicx}
\usepackage{color}
\usepackage{graphics}
\usepackage[latin1]{inputenc}
\usepackage{verbatim}
\usepackage{tikz}
\usepackage{multirow}
\usepackage{expdlist}
\usepackage{subfigure}

\numberwithin{equation}{section}

\newtheorem{theorem}{Theorem}[section]

\newtheorem{proposition}[theorem]{Proposition}
\newtheorem{definition}[theorem]{Definition}
\newtheorem{corollary}[theorem]{Corollary}
\newtheorem{remark}[theorem]{Remark}
\newtheorem{example}[theorem]{Example}

\newcommand{\la}{\lambda}

\newcommand{\RL}{\operatorname{RL}}
\newcommand{\id}{\operatorname{id}}
\newcommand{\RLk}[1]{\RL_{#1}\text{--}\min}
\newcommand{\Sn}{\mathcal{S}}
\newcommand{\stirlingone}[2]{\genfrac{[}{]}{0pt}{}{#1}{#2}}
\newcommand{\evalat}[1]{\,\bigr\rvert_{#1}}
\def\wrt{\mbox{w.r.t. }}
\def\ie{\mbox{i.e. }}
\providecommand{\abs}[1]{\left\lvert#1\right\rvert}

\begin{document}

\author{Lukas Riegler \and Christoph Neumann}
\title[Playing JDT on d-complete posets]{Playing jeu de taquin on d-complete posets}
\thanks{E-Mail: lukas.riegler@univie.ac.at. Supported by the Austrian Science Foundation FWF, START grant Y463.}
\thanks{E-Mail: christoph.neumann@univie.ac.at. Supported by the Austrian Science Foundation FWF, Wittgenstein grant Z130-N13.}
\address{Fakultät für Mathematik, Universität Wien, Oskar-Morgenstern-Platz 1, 1090 Wien, Austria}

\begin{abstract}
Using a modified version of jeu de taquin, Novelli, Pak and Stoyanovskii gave a bijective proof of the hook-length formula for counting standard Young tableaux of fixed shape. In this paper we consider a natural extension of jeu de taquin to arbitrary posets. Given a poset $P$, jeu de taquin defines a map from the set of bijective labelings of the poset elements with $\{1,2,\ldots,\abs{P}\}$ to the set of linear extensions of the poset. One question of particular interest is for which posets this map yields each linear extension equally often. We analyze the double-tailed diamond poset $D_{m,n}$ and show that uniform distribution is obtained if and only if $D_{m,n}$ is d-complete. Furthermore, we observe that the extended hook-length formula for counting linear extensions on d-complete posets provides a combinatorial answer to a seemingly unrelated question, namely: Given a uniformly random standard Young tableau of fixed shape, what is the expected value of the left-most entry in the second row?
\end{abstract}

\maketitle

\section{Introduction}
\label{sec:Introduction}

\subsection{Jeu de taquin on posets}
Jeu de taquin (literally translated 'teasing game') is a board game (also known as $15$-puzzle) where
fifteen square tiles numbered with $\{1,2,\ldots,15\}$ are arranged inside a $4
\times 4$ square. The goal of the game is to sort the tiles by consecutively
sliding a square into the empty spot (see Figure \ref{fig:jdt-boardgame}). In
combinatorics the concept of jeu de taquin was originally introduced by
Schützenberger \cite{SchuetzenbergerJDT} on skew standard Young tableaux. Two related operations called promotion and evacuation, which act bijectively on the set of linear extensions of a poset, were also defined by Schützenberger \cite{Schuetzenberger72Promotion}.
\begin{figure}[ht]
\centering
\includegraphics[width=0.15\textwidth]{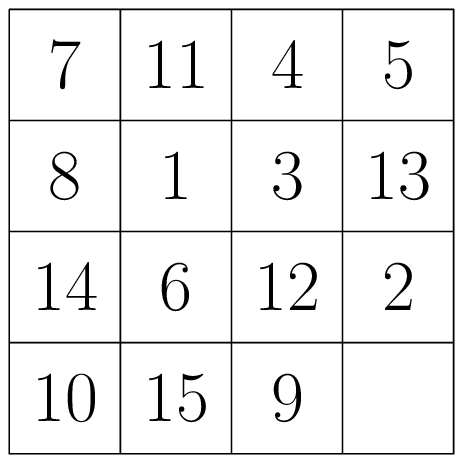}
\quad
\includegraphics[width=0.15\textwidth]{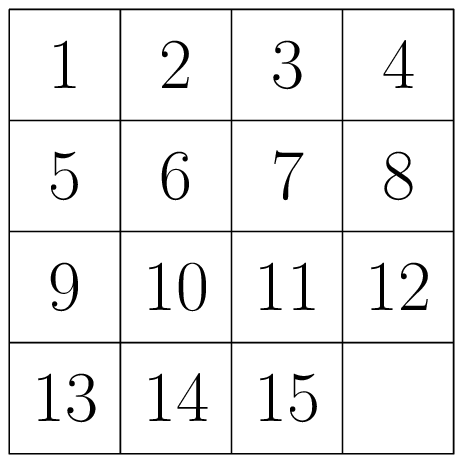}
\caption{An initial and final configuration of the board game jeu de taquin.
\label{fig:jdt-boardgame}}
\end{figure}
A modified version of jeu de taquin \cite{NPSHookBijection} has an obvious extension to arbitrary posets, which we describe first: The goal
of jeu de taquin on an $n$-element poset $P$ is to transform any (bijective)
labeling of the poset elements with $[n]:=\{1,2,\ldots,n\}$ into a dual linear extension, \ie a labeling $\iota$ such that $\iota(x) > \iota(y)$ whenever $x < y$ in the poset. For this, we first fix a linear extension $\sigma:P\rightarrow [n]$ of the poset, which defines the order in which the labels are sorted (see Figure \ref{fig:poset-labeling}).
\begin{figure}[ht]
\begin{center}
\includegraphics[width=4cm]{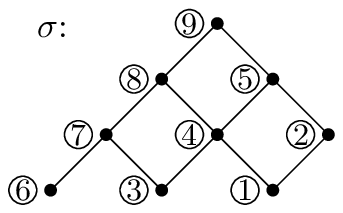}
\caption{Hasse diagram of a poset and a linear extension $\sigma$.\label{fig:poset-labeling}}
\end{center}
\end{figure}
The sorting procedure consists of $n$ rounds where after the first $i$ rounds
the poset elements $\{\sigma^{-1}(1),\sigma^{-1}(2),\ldots,\sigma^{-1}(i)\}$
have dually ordered labels. To achieve this, we compare in round $i$ the
current label of $x:=\sigma^{-1}(i)$ with the labels of all poset elements covered by $x$. If the current label is the smallest of them, we are done with round $i$. Else, let $y$ be the poset element with the smallest label. Swap the labels of $x$ and $y$ and repeat with the new label of $y$. An example can be seen in Figure \ref{fig:jdt-alg}. The fact that $\sigma$ is a linear extension together with the minimality condition in the sorting procedure ensures that after $i$ rounds the poset elements $\{\sigma^{-1}(1),\sigma^{-1}(2),\ldots,\sigma^{-1}(i)\}$ have dually ordered labels. In particular, the sorting procedure transforms each labeling of the poset into a dual linear extension.
\begin{figure}[ht]
\begin{center}
\raisebox{-0.5\height}{\includegraphics[width=3.5cm]{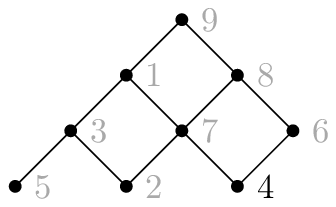}}
$\leadsto$
\raisebox{-0.5\height}{\includegraphics[width=3.5cm]{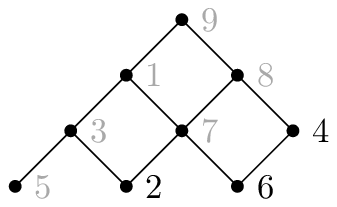}}
$\leadsto$
\raisebox{-0.5\height}{\includegraphics[width=3.5cm]{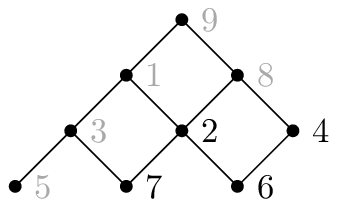}} \\
$\leadsto$
\raisebox{-0.5\height}{\includegraphics[width=3.5cm]{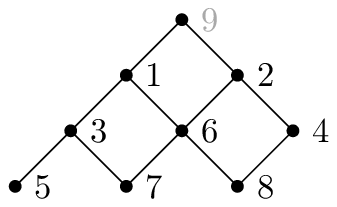}}
$\leadsto$
\raisebox{-0.5\height}{\includegraphics[width=3.5cm]{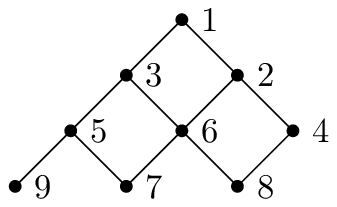}}
\caption{Example of jeu de taquin in order $\sigma$ as given in Figure \ref{fig:poset-labeling}.\label{fig:jdt-alg}}
\end{center}
\end{figure}
The question we are interested in is: Given a uniformly random labeling of the poset
elements, does jeu de taquin output a uniformly random dual linear extension of the poset? More specifically, given a poset, is there an order $\sigma$ such that playing jeu de taquin with all possible labelings yields each dual linear extension equally often? If yes, then jeu de taquin allows us to immediately extend each algorithm for creating uniformly random permutations to an algorithm creating uniformly random linear extensions of the poset.

\subsection{(Shifted) standard Young tableaux \& the hook-length formula }
For certain classes of posets we know the answer: Most famously, the Young diagram of an integer partition $\la = (\la_1,\la_2,\ldots,\la_k)\vdash n$ can be considered as a poset (see Figure \ref{fig:young-diagram-poset}).
\begin{figure}[ht]
\begin{center}
\includegraphics[width=8cm]{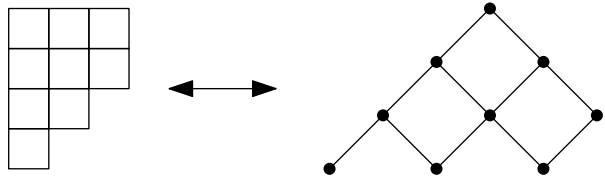}
\caption{Young diagram of the partition $(3,3,2,1)$ and the corresponding
poset.\label{fig:young-diagram-poset}}
\end{center}
\end{figure}
A Young tableau is a bijective filling of the boxes with $[n]$ and thus corresponds to a labeling of the poset. A standard Young tableau is a filling of the boxes where entries in each row (left-to-right) and column (top-down) are strictly increasing. Hence, standard Young tableaux correspond to the dual linear extensions of the respective poset. Novelli, Pak and Stoyanovskii gave a bijective proof \cite{NPSHookBijection} of the fact that jeu de taquin with column-wise order $\sigma$ (as in Figure \ref{fig:poset-labeling}) yields uniform distribution among standard Young tableaux. Their bijective proof can actually be extended to work for orders different from column-wise order (see also \cite{SaganSymmetricGroup}). 

A second class of posets where we know the answer corresponds to shifted Young diagrams of strict integer partitions, \ie partitions where $\la_1 > \la_2 > \dots > \la_k$. In this case, the boxes of the shifted Young diagram are indented (see Figure \ref{fig:shifted-syt-poset}).
\begin{figure}[ht]
\begin{center}
\includegraphics[width=8cm]{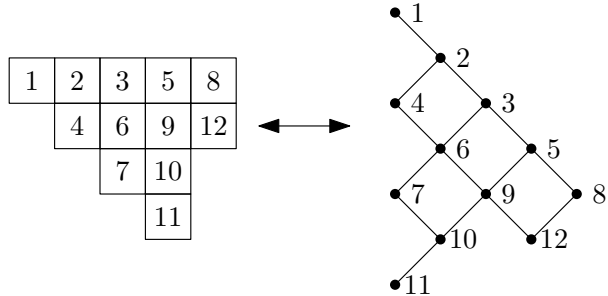}
\caption{A shifted standard Young tableau of shape $(5,4,2,1)$ and the dual linear extension of the corresponding poset.\label{fig:shifted-syt-poset}}
\end{center}
\end{figure}
It was shown by Fischer \cite{FischerShiftedHookBijection} that row-wise order $\sigma$ yields uniform distribution among shifted standard Young tableaux (however column-wise order fails for the partition $(4,3,2,1)$). 

What both classes have in common is that the number of different standard fillings of fixed shape $\la$ can be obtained with a simple product formula, called hook-length formula: In the case of Young diagrams the hook of a cell consists of all cells to the right in the same row, all cells below in the same column and the cell itself. The hook-length $h_c$ of a cell $c$ is the number of cells in its hook (see Figure \ref{fig:hook-lengths}).
\begin{figure}[ht]
\centering
\subfigure[]{%
\includegraphics[height=2.5cm]{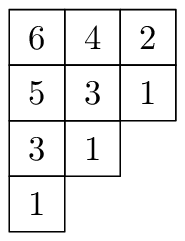}
\label{fig:hook-lengths}}
\quad
\subfigure[]{%
\includegraphics[height=2.5cm]{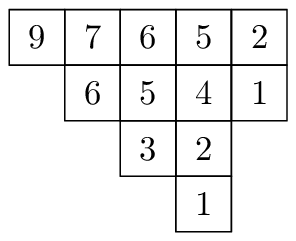}
\label{fig:shifted-hook-lengths}}
\caption{The hook-lengths of all cells in a Young diagram and a shifted Young diagram.}
\label{fig:hook-length-overview}
\end{figure}
The number $f^\la$ of standard Young tableaux of fixed shape $\la$ is then given by (\cite{FRTHookFormula}, \cite{GansnerShiftedHookFormula})
\begin{equation}
\label{eq:FRTHookFormula}
f^\la = \frac{n!}{\prod\limits_{c \in \la} h_c},
\end{equation}  
where the product is taken over all cells $c$ in the Young diagram. For shifted
Young diagrams the number of standard fillings can be obtained with the same hook-formula \eqref{eq:FRTHookFormula} by modifying the definition of the hooks: the shifted hook of a cell contains the same cells as the hook before, and additionally, if the hook contains the left-most cell in row $i$, then the shifted hook is extended to all cells in row $i+1$ (see Figure \ref{fig:shifted-hook-lengths}).

\subsection{d-complete posets}
The definition of the hook-lengths can be generalized so that the hook-length formula extends to further classes of posets, called d-complete posets \cite{ProctorClassification}.

For $m,n \geq 2$ the poset $D_{m,n}$
consists of $m+n$ elements for which the Hasse diagram is obtained by taking a diamond of four elements and appending a chain of $m-2$ elements at
the top element of the diamond and a chain of $n-2$ elements at the bottom
element of the diamond (see Figure \ref{fig:dmn-poset}). The poset $D_{m,n}$ is referred to as double-tailed diamond.
\begin{figure}[ht]
\begin{center}
\includegraphics[height=5cm]{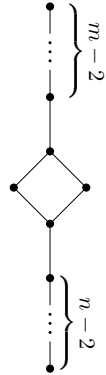}
\caption{Hasse diagram of the double-tailed diamond $D_{m,n}$.\label{fig:dmn-poset}}
\end{center}
\end{figure}
As elementary building blocks the double-tailed diamonds play a fundamental role in the definition of d-complete posets \cite{ProctorClassification}: Given a poset $P$ and $k \geq 3$, an interval $[w,z]$ in the poset is called $d_k$-interval if $[w,z] \cong D_{k-1,k-1}$. An interval $[w,y]$ is called $d_k^-$-interval if $[w,y] \cong D_{k-2,k-1}$ (in the special case $k=3$ let us abuse notation and say that a $d_3^-$-interval is a diamond with top element removed). A poset $P$ is called $d_k$-complete if it satisfies the following three conditions:
\begin{enumerate}
  \item $[w,y]$ is $d_k^-$-interval $\Rightarrow$ $\exists z \in P: [w,z]$ is $d_k$-interval,
  \item $[w,z]$ is $d_k$-interval $\Rightarrow$ $z$ does not cover an element outside of $[w,z]$ and
  \item $[w,z]$ is $d_k$-interval $\Rightarrow$ there exists no $w' \neq w$ such that $[w',z]$ is $d_k$-interval.
\end{enumerate} 
A poset $P$ is called d-complete if and only if $P$ is $d_k$-complete for all $k\geq 3$. The posets corresponding to Young diagrams and shifted Young diagrams are examples of d-complete posets (see Figure \ref{fig:dcomplete-examples}).
\begin{figure}[ht]
\begin{center}
\includegraphics[height=4cm]{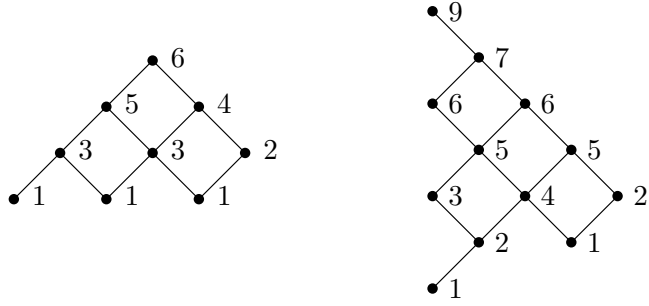}
\caption{Two d-complete posets with assigned hook-lengths.\label{fig:dcomplete-examples}}
\end{center}
\end{figure}
A full classification of d-complete posets can be found in \cite{ProctorClassification}.

Recall that for any poset $P$ a map $\sigma: P \to \mathbb{N}$ is called $P$-partition of $n$ if $\sigma$ is order-reversing and satisfies $\sum_{x\in P} \sigma(x)=n$. Let $G_P(x)$ denote the  corresponding generating function, \ie
\[
G_P(x):=\sum_{n \geq 0} a_n x^n,
\]
where $a_n$ denotes the number of $P$-partitions of $n$. As a result by R.P.~Stanley \cite{StanleyEnumComb1}*{Theorem 3.15.7} the generating function can be factorized into
\begin{equation}
\label{eq:StanleyPGenFun}
G_P(x)=\frac{W_P(x)}{(1-x)(1-x^2)\dots(1-x^{\abs{P}})}
\end{equation}
with a polynomial $W_P(x)$ such that $W_P(1)$ is the number of linear extensions of $P$. A poset $P$ is called hook-length poset if there exists a map $h:P \to \mathbb{Z}^+$ such that
\begin{equation}
\label{eq:HookLengthFactorisation}
G_P(x)=\prod_{z \in P} \frac{1}{1-x^{h(z)}}.
\end{equation}
The number $f^P$ of linear extensions of a hook-length poset can be obtained by equating \eqref{eq:StanleyPGenFun} and \eqref{eq:HookLengthFactorisation}, and taking the limit $x \to 1$:
\begin{equation}
\label{eq:posetHLF}
f^P = \frac{\abs{P}!}{\prod_{z \in P} h(z)}.
\end{equation}

Every d-complete poset is a hook-length poset \cite{ProctorWebsite}. In fact, d-complete posets were generalized to so-called leaf posets \cite{IshikawaTagawaLeafPosets}, which are also hook-length posets. The hook-lengths $h_z:=h(z)$ for d-complete posets can be obtained in the following way:  
\begin{enumerate}
  \item Assign all minimal elements of the poset the hook-length $1$.
  \item Repeat until all elements have their hook-length assigned: Choose a poset element $z$ where all smaller elements have their hook-length assigned. Check whether $z$ is the top element of a $d_k$-interval $[w,z]$. 
  \begin{itemize}
    \item If no, set $h_z := \#\{ y \in P : y \leq z\}$.
    \item If yes, set $h_z := h_l+h_r-h_w$, where $l$ and $r$ are the two incomparable elements of the double-tailed diamond $[w,z]$.
  \end{itemize}       
\end{enumerate}
By definition of d-complete posets the procedure is well-defined (there exists at most one $d_k$-interval with $z$ as top element). Moreover, it is a nice exercise to check that this definition is equivalent to the previous definition of hook-lengths for Young diagrams (which only contain $D_{2,2}$ intervals) and shifted Young diagrams (which additionally contain $D_{3,3}$ intervals along the left rim). As an example compare Figure \ref{fig:hook-length-overview} and Figure \ref{fig:dcomplete-examples}.

\subsection{Jeu de taquin on the double-tailed diamond}
Since there is exactly one pair of incomparable elements in the double-tailed diamond $D_{m,n}$, there are two different dual linear extensions $T_1$ and $T_2$ of $D_{m,n}$ (see Figure \ref{fig:dmn-poset-dual-linear-extension}).
\begin{figure}[ht]
\begin{center}
\includegraphics[height=6cm]{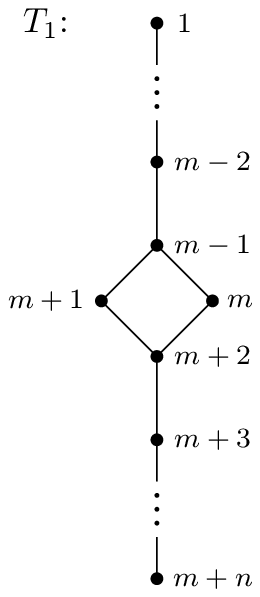}
\qquad\qquad
\includegraphics[height=6cm]{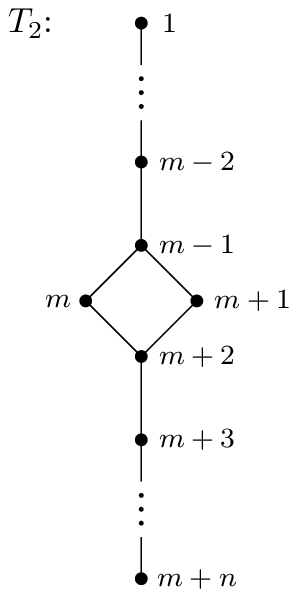}
\caption{The two possible dual linear extension of $D_{m,n}$.\label{fig:dmn-poset-dual-linear-extension}}
\end{center}
\end{figure}
For jeu de taquin we choose w.l.o.g. the order $\sigma$ that corresponds to the reverse order of $T_1$. In Section \ref{sec:JDTonDmn} we show that jeu de taquin yields uniform distribution if and only if $m \geq n$. We proceed by defining a related statistic on permutations generalising right-to-left minima. In terms of this statistic we can analyze a refined counting problem, namely counting the number of permutations for which jeu de taquin swaps the order between the labels of the two incomparable elements exactly $k$ times. As it turns out this counting problem has a nice closed solution (Proposition \ref{prop:RLmin}) as well as the resulting difference between the number of permutations yielding $T_1$ and $T_2$:
\begin{theorem}
\label{th:mainTheorem}
Let $s_{m,n}^{(1)}$ (resp. $s_{m,n}^{(2)}$) denote the number of permutations in
$\Sn_{m+n}$ which jeu de taquin on $D_{m,n}$ with order $\sigma$ maps to $T_1$
(resp.
$T_2$). Then
\begin{equation}
s_{m,n}^{(1)} - s_{m,n}^{(2)} = (-1)^m\binom{n-1}{m} m!\; n!, \quad m,n\geq 2.
\end{equation}
In particular, $s_{m,n}^{(1)} = s_{m,n}^{(2)}$ if and only if $m \geq n$.    
\end{theorem}
What is interesting about the result is that the poset $D_{m,n}$ is d-complete
if and only if $m \geq n$. Together with Young diagrams and shifted Young
diagrams (two further classes of d-complete posets) this hints towards a
connection between d-completeness of a poset and the property that jeu de
taquin \wrt an appropriate order yields uniform distribution.

In Section \ref{sec:BijectiveProofOfMT} we give a purely combinatorial proof of
Theorem \ref{th:mainTheorem} by constructing an appropriate involution
$\Phi_{m,n}$ on $\Sn_{m+n}$ if $m \geq n$. In the case $m < n$ we identify a set
$\mathcal{E}$ of of $\binom{n-1}{m} m!\; n!$ exceptional permutations and construct an appropriate involution $\Phi_{m,n}$ on $\Sn_{m+n}
\setminus \mathcal{E}$.

\subsection{Jeu de taquin on insets}
The class of insets (fourth class in the classification of d-complete posets in
\cite{ProctorClassification}) can be defined in terms of the shape of its
corresponding diagram. For $k \geq 2$ and $\la = (\la_1,\ldots,\la_k)\vdash n$
the inset $P_{k,\la}$ is obtained by taking the Young diagram corresponding to
$\la$ and adding $k-1$ boxes at the left end of the first row and one box at
the left end of the second row (see Figure \ref{fig:p4-3,2,2,1}).
\begin{figure}[ht]
\begin{center}
\includegraphics[height=3cm]{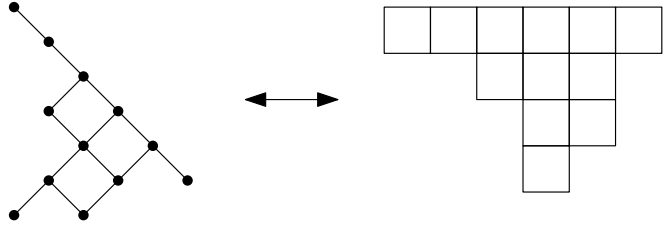}
\caption{The inset $P_{4,(3,2,2,1)}$ and its corresponding box diagram.\label{fig:p4-3,2,2,1}}
\end{center}
\end{figure}
The hook-lengths of the cells in $\la$ can be computed like for Young diagrams.
The additional box in the second row is not the maximum of a double-tailed
diamond interval, whereas each of the $k-1$ additional boxes in the first row is the top element of a double-tailed diamond interval. The resulting hook-lengths are
depicted in Figure \ref{fig:inset-hooklengths}.
\begin{figure}[ht]
\begin{center}
\includegraphics[height=5cm]{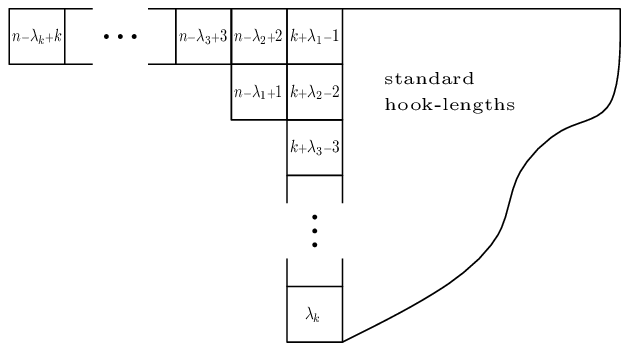}
\caption{The hook-lengths of $P_{k,\la}$ with $\la=(\la_1,\ldots,\la_k)\vdash n$.\label{fig:inset-hooklengths}}
\end{center}
\end{figure}
The hook-length formula \eqref{eq:posetHLF} implies that the number $f^{k,\la}$
of standard fillings is given by
\begin{equation}
\label{eq:HLFInsets}
f^{k,\la} = \frac{(n+k)!}{\left( \prod\limits_{c\in\la} h_c \right) \left( \prod\limits_{i=1}^k (n - \la_i + i) \right)}.
\end{equation}
Computational experiments indicate that jeu de taquin on $P_{k,\la}$ with
row-wise order again yields uniform distribution. Even though we were so far not
able to modify the techniques of \cite{NPSHookBijection} and
\cite{FischerShiftedHookBijection} to prove that jeu de taquin indeed yields
uniform distribution, a quick analysis of insets yields a solution to a different, nice problem:

Fix an integer partition $\la=(\la_1,\ldots,\la_k)$ and consider uniform
distribution on the set of standard Young tableaux of shape $\la$. What is the
expected value of the left-most entry in the second row? Three examples are depicted in Figure \ref{fig:expectation-examples}.
\begin{figure}[ht]
\begin{center}
\includegraphics[height=3cm]{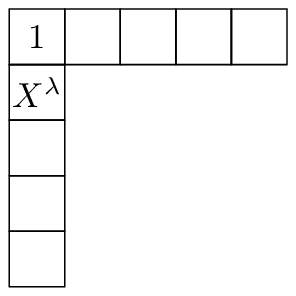}
\qquad
\includegraphics[height=3cm]{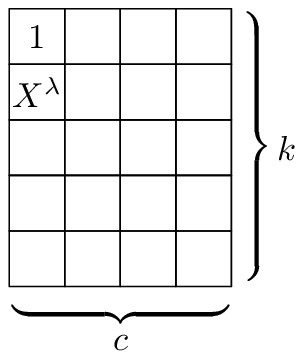}
\qquad
\includegraphics[height=3cm]{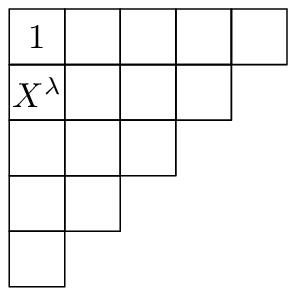}
\caption{What is $\mathbb{E}X^\la$ under uniform distribution among standard Young tableaux?\label{fig:expectation-examples}}
\end{center}
\end{figure}
In Section \ref{sec:CorollaryExpectedValue} we prove the answer given in
Corollary \ref{cor:mainCorollary}.
\begin{corollary}
\label{cor:mainCorollary}
Fix a partition $\la=(\la_1,\ldots,\la_k) \vdash n$. Let $(\Omega,2^\Omega,P)$ be the probability space containing the $f^\lambda$ different standard Young tableaux of shape $\la$ and uniform probability measure $P$. Let $X^\la \in \{2,3,\ldots,\la_1+1\}$ denote the random variable measuring the left-most entry in the second row. Then  
\begin{equation}
\label{eq:expectProduct}
\mathbb{E}X^\la = \frac{f^{k,\la}}{f^\la} = \prod_{i=1}^{k} \frac{n+i}{n+i-\la_{i}}.
\end{equation}
\end{corollary}
Take for example the partition $(3,3,2,1)$ from Figure \ref{fig:hook-lengths}
which has according to the hook-length formula $\frac{9!}{6\cdot 5 \cdot 4
\cdot 3 \cdot 3 \cdot 2}=168$ different standard Young tableaux. The corresponding inset has $f^{4,(3,3,2,1)}=429$ standard fillings. Hence, \eqref{eq:expectProduct} tells us that in a standard Young tableau of shape $(3,3,2,1)$ the left-most entry in the second row is on average $\frac{429}{168}\approx 2.554.$

The expected value could also be expressed as a sum of determinants by Aitken's determinant formula for skew standard Young tableaux \citelist{\cite{AitkenDeterminantFormula} \cite{StanleyEnumComb2}*{Corollary 7.16.3}}. While it should be possible to derive the same product formula \eqref{eq:expectProduct} from this expression, the simplicity of the combinatorial argument given in Section \ref{sec:CorollaryExpectedValue} is somewhat appealing.

A different approach for generating uniformly random linear extensions was taken by Nakada
and Okamura \citelist{\cite{NakadaOkamuraLinExtAlgSimplyLaced} \cite{NakadaLinExtAlgNonSimplyLaced}}. They use a probabilistic algorithm for
generating linear extensions and compute the probability $p(L)$ that a fixed
linear extension $L$ is generated by the algorithm. Since they show that $p(L)$
actually does not depend on $L$ their statement not only implies that the
algorithm yields uniform distribution among linear extensions but also that the number of linear extensions is given by $\frac{1}{p}$.

\section{Jeu de taquin on the double-tailed diamond}
\label{sec:JDTonDmn}

\subsection{Reducing the problem to understanding a permutation statistic}
For the purpose of this section let us visualize the elements of the
double-tailed diamond $D_{m,n}$ as boxes and labelings as fillings of the
boxes. Let $B_{i,j}$ denote the box in row $i$ and column $j$ and given a
filling of the boxes let $T_{i,j}$ denote the entry in box $B_{i,j}$ (see Figure
\ref{fig:d65-box-coordinates}).
\begin{figure}[ht]
\begin{center}
\includegraphics[height=3cm]{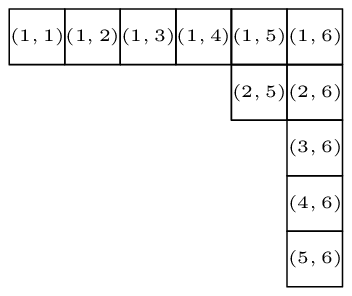}
\quad
\includegraphics[height=3cm]{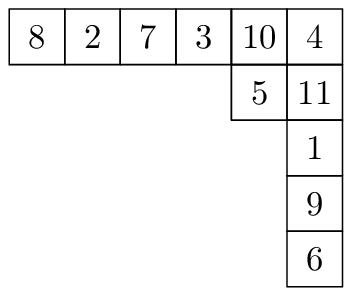}
\caption{Coordinates of the boxes in $D_{6,5}$
and a filling.\label{fig:d65-box-coordinates}}
\end{center}
\end{figure}
We perform modified jeu de taquin on $D_{m,n}$ with respect to the linear
extension $\sigma$ satisfying $\sigma(B_{2,m-1}) = n$ and $\sigma(B_{1,m})=n+1$.
Given a permutation $\pi = \pi_1\pi_2\dots\pi_{m+n}\in \mathcal{S}_{m+n}$ we start jeu de taquin by assigning $(\pi_1$, $\pi_2$, $\ldots$, $\pi_{m+n})$ to the boxes in reverse order of $\sigma$ (see Figure \ref{fig:d65-initial-filling}).
\begin{figure}[ht]
\begin{center}
\includegraphics[height=3cm]{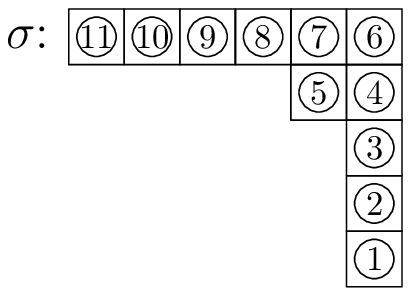}
\quad
\includegraphics[height=3cm]{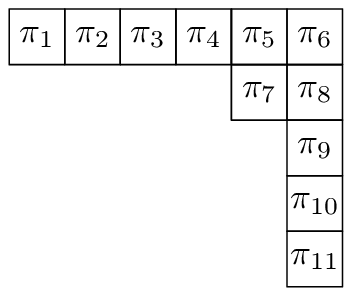}
\caption{Linear extension $\sigma$ for jeu de taquin and initial filling of the boxes.\label{fig:d65-initial-filling}}
\end{center}
\end{figure}

Let $x_i:=x_i(\pi)$ (resp. $y_i:=y_i(\pi)$) denote the entry $T_{1,m}$ (resp. $T_{2,m-1}$) after $i$ rounds of jeu de taquin. So, the initial values are $x_0=\pi_m$ and $y_0=\pi_{m+1}$ and we know that in the end we have $\{x_{m+n},y_{m+n}\}=\{m,m+1\}$. As in Theorem \ref{th:mainTheorem} we denote by $s_{m,n}^{(1)}$ the number of permutations $\pi \in \mathcal{S}_{m+n}$ with $x_{m+n}(\pi)=m$ and by $s_{m,n}^{(2)}$ the number of permutations with $x_{m+n}(\pi)=m+1$.

In the first $n$ rounds of jeu de taquin the elements $\{\pi_{m+n},\pi_{m+n-1},
\ldots, \pi_{m+1}\}$ are simply sorted in increasing order (cf. Insertion-Sort
algorithm). Therefore
\[
x_n(\pi) = \pi_m \quad\text{and}\quad y_n(\pi) =
\min\{\pi_{m+1},\pi_{m+2},\ldots,\pi_{m+n}\}.
\]
In the following we are no longer interested in the exact values of $x_i$ and $y_i$ but only whether $x_i < y_i$ or $x_i > y_i$:
If $x_n < y_n$, then $x_n < T_{2,m}$, so nothing happens in the $(n+1)$-st round
of jeu de taquin and $x_{n+1} < y_{n+1}$. If on the other hand $x_n > y_n$, then
$T_{1,m}$ may or may not be swapped with $T_{2,m}$ (and further entries), but
in any case $x_{n+1} > y_{n+1}$. Therefore, $x_{n+1}(\pi) < y_{n+1}(\pi)$ if and
only if $\pi_m=\min\{\pi_m,\pi_{m+1},\ldots,\pi_{m+n}\}$, \ie if $\pi_m$ is a right-to-left minimum of $\pi$.

For the remaining rounds we observe the following: Before moving $\pi_i$ at the start of round $m+n+1-i$ the boxes $B_{1,i+1},\ldots,B_{1,m},B_{2,m-1}$ contain the $m+1-i$ smallest elements of $\{\pi_{i+1},\pi_{i+2},\ldots,\pi_{m+n}\}$. We have to distinguish between two cases, namely whether $\pi_i$ is among the $m+1-i$ smallest elements of $\{\pi_{i},\pi_{i+1},\ldots,\pi_{m+n}\}$ or not.

If on the one hand $\pi_i > \max\{x_{m+n-i},y_{m+n-i}\}$, then
jeu de taquin first moves $\pi_i$ to $B_{1,m-1}$ and then swaps $\pi_i$ with $\min\{x_{m+n-i},y_{m+n-i}\}$, which -- by assumption -- changes  the order between $T_{1,m}$ and $T_{2,m-1}$. After that $\pi_i$ may or may not move further, but in any case the order between $x_{m+n-i+1}$ and $y_{m+n-i+1}$ is exactly the opposite of the order between $x_{m+n-i}$ and $y_{m+n-i}$. If on the other hand $\pi_i < \max\{x_{m+n-i},y_{m+n-i}\}$, then jeu de taquin
moves $\pi_i$ at most to $B_{1,m}$ or $B_{2,m-1}$ and -- by assumption -- does
not change the order. Thus $x_{m+n-i+1}$ and $y_{m+n-i+1}$ are in the same order as $x_{m+n-i}$ and $y_{m+n-i}$.

Summed up, we have observed that $x_{n+1}(\pi) < y_{n+1}(\pi)$ if and only if $\pi_m$ is a right-to-left minimum. After that, the order between $T_{1,m}$ and $T_{2,m-1}$ is kept the same in the round starting with $\pi_i$ if and only if $\pi_i$ is among the $m+1-i$ smallest elements of
$\{\pi_i,\pi_{i+1},\ldots,\pi_{m+n}\}$. Therefore, we have reduced the problem to understanding a corresponding statistic on permutations.

\subsection{Definition and analysis of a statistic on permutations}

The previous observations motivate the following definition generalizing right-to-left-minima of permutations: 
\begin{definition}[$\RLk{k}$] Let $\pi=\pi_1 \pi_2\dots\pi_n \in \mathcal{S}_n$. We say that $\pi_i$ is a $\RLk{k}$ if and only if $\pi_i$ is among the $k$ smallest elements of $\{\pi_i,\pi_{i+1},\ldots,\pi_n\}$.
\end{definition}
To solve our counting problem, we need to understand the distribution of
\begin{equation}
c_{m,n}(\pi):=\sum_{i=1}^m \left[\pi_i \text{ is } \RLk{m+1-i} \right], \quad \pi=\pi_1\cdots\pi_{m+n} \in \Sn_{m+n},
\end{equation}
where the square brackets denote Iverson brackets, \ie $[\phi]:=1$ if $\phi$ is
true, and $0$ otherwise. As it turns out the distribution of $c_{m,n}$ can be simply expressed:
\begin{proposition}
\label{prop:RLmin}
Let
\[
c_{m,n,k}:=\abs{\{\pi \in \Sn_{m+n}: c_{m,n}(\pi) = m-k\}}, \quad m,n \geq 1,\; 0 \leq k \leq m.
\]
Then
\begin{equation}
c_{m,n,k} = n^k \stirlingone{m+1}{k+1} n!,
\end{equation}
where $\stirlingone{s}{t}$ denotes the unsigned Stirling numbers of first kind, \ie the number of permutations of $s$ elements with $t$ disjoint cycles. 
\end{proposition}
\begin{remark}
\label{rem:parity}
From the previous observations it follows that $c_{m,n,k}$ counts the number of $\pi \in \Sn_{m+n}$ for which the order between $T_{1,m}$ and $T_{2,m-1}$ in jeu de taquin is changed exactly $k$ times (with $\pi_m$ contributing to $k$ if and
 only if $\pi_m$ is not a right-to-left minimum). In particular, $x_{m+n}(\pi) < y_{m+n}(\pi)$ if and only if $k=m - c_{m,n}(\pi)$ is even.
\end{remark}
\begin{proof}[Proof of Proposition \ref{prop:RLmin}]
The proof is split into the two edge cases $k=m$ and $k=0$ and the case $0 < k < m$.
\begin{itemize}
   \item $k=0$: Let us show that
   \[
  c_{m,n}(\pi) = m \quad\iff\quad \{\pi_1,\ldots,\pi_m\}=\{1,\ldots,m\}.
  \]
  Assume $s(\pi_i):=\left[\pi_i \text{ is } \RLk{m+1-i} \right]=1$ for all
  $i=1,\ldots,m$. Then $s(\pi_m)=1$ implies $1 \in \{\pi_1,\ldots,\pi_m\}$.
  Suppose there exists $2 \leq k \leq m$ such that $k \in
  \{\pi_{m+1},\ldots,\pi_{m+n}\}$. It then follows from $s(\pi_m) = s(\pi_{m-1})
  = \dots = s(\pi_{m+2-k})=1$ that none of $\pi_{m},\pi_{m-1},\ldots,\pi_{m+2-k}$ can be greater than $k$, \ie
  $\{\pi_{m+2-k},\ldots,\pi_{m}\}=\{1,2,\ldots,k-1\}$. But this contradicts $\pi_{m+1-k}$ being a $\RLk{k}$. The reverse direction is obvious. Since there
  are exactly $m!$ permutations in $\mathcal{S}_{m+1}$ consisting of one cycle, we obtain
  \[
  c_{m,n,0} = m!\; n! = n^0 \stirlingone{m+1}{1} n!.
  \]
  \item $k=m$: In this case let us observe that
  \[
  c_{m,n}(\pi) = 0 \quad\iff\quad j \in \{\pi_{m+2-j},\pi_{m+3-j},\ldots,\pi_{m+n}\}\quad \text{for all }  j=1,\ldots,m.
  \]
  Suppose $j \notin \{\pi_{m+2-j},\pi_{m+3-j},\ldots,\pi_{m+n}\}$. Then there
  exists an $i \in \{1,\ldots,m+1-j\}$ such that $\pi_i = j \leq m+1-i$. But this
  implies that $\pi_i$ is a $\RLk{m+1-i}$, so $c_{m,n}(\pi) > 0$. If we conversely
  assume that $\{1,2,\ldots,m+1-i\} \subseteq \{\pi_{i+1},\ldots,\pi_{m+n}\}$
  for all $i=1,\ldots,m$, it follows that $\pi_i$ is not among the $m+1-i$ smallest elements of $\{\pi_i,\ldots,\pi_{m+n}\}$, and therefore $c_{m,n}(\pi)=0$.
  
  Therefore we have exactly $n^m$ possibilities to choose the preimage of
  $\{1,2,\ldots,m\}$ and for each such choice the preimages of $\{m+1,\ldots,m+n\}$ can be chosen in any order, \ie
  \[
  c_{m,n,m} = n^m n! = n^m \stirlingone{m+1}{m+1} n!.
  \]  
  \item $0 < k < m$:  We proceed by induction on $m$. From the recurrence relation $\stirlingone{s+1}{t}=s\stirlingone{s}{t}+\stirlingone{s}{t-1}$  and the induction hypothesis (resp. the edge cases) it follows that
  \[
  n^k \stirlingone{m+1}{k+1}n! = n^k m \stirlingone{m}{k+1}n!+n^k\stirlingone{m}{k}n! = m\; c_{m-1,n,k}+n\; c_{m-1,n,k-1}.
  \]
  So it only remains to show that
  \begin{equation}
  \label{eq:cmnk-recursion}
  m \; c_{m-1,n,k}+ n\; c_{m-1,n,k-1} = c_{m,n,k}.
  \end{equation}
  For a bijective proof of \eqref{eq:cmnk-recursion} consider the position of
  $1$ in $\pi_1\dots\pi_m\pi_{m+1}\dots\pi_{m+n}$. Let $\pi'$ denote the
  permutation $\pi$ with $\pi_j=1$ removed and each number reduced by $1$, so
  that $\pi' \in \mathcal{S}_{m+n-1}$.
  
  If on the one hand $1\leq j \leq m$, note
  that $c_{m-1,n}(\pi')=c_{m,n}(\pi)-1$ and thus $c_{m-1,n}(\pi')=(m-1)-k$ if and only if
  $c_{m,n}(\pi)=m-k$. For each $1 \leq j \leq m$ this establishes a one-to-one
  correspondence between the permutations $\pi\in\mathcal{S}_{m+n}$ with
  $\pi_j=1$ that are counted by $c_{m,n,k}$ and permutations
  $\pi'\in\mathcal{S}_{m+n-1}$ counted by $c_{m-1,n,k}$.
  
  If on the other hand $m+1 \leq j \leq m+n$, note that $c_{m-1,n}(\pi')=c_{m,n}(\pi)$ and
  therefore $c_{m-1,n}(\pi')=(m-1)-(k-1)$ if and only if $c_{m,n}(\pi)=m-k$. For each $m+1
  \leq j \leq m+n$ this is a one-to-one correspondence between the permutations $\pi\in\mathcal{S}_{m+n}$ with
  $\pi_j=1$ that are counted by $c_{m,n,k}$ and permutations
  $\pi'\in\mathcal{S}_{m+n-1}$ counted by $c_{m-1,n,k-1}$.
\end{itemize}
\end{proof}

\begin{proof}[Proof of Theorem \ref{th:mainTheorem}]
In Remark \ref{rem:parity} we have observed that
\[
s_{m,n}^{(1)} - s_{m,n}^{(2)} = \abs{\{\pi \in \Sn_{m+n}: m - c_{m,n}(\pi) \text{
is even}\}}-\abs{\{\pi \in \Sn_{m+n}: m - c_{m,n}(\pi) \text{ is odd}\}}.
\]
Together with Proposition \ref{prop:RLmin} it follows that
\begin{align*}
s_{m,n}^{(1)} - s_{m,n}^{(2)} &= \sum_{k=0}^m (-1)^k c_{m,n,k} = \sum_{k=1}^{m+1} (-1)^{k-1} n^{k-1} \stirlingone{m+1}{k} n! \\
&= (-1)^m (n-1)! \sum_{k=0}^{m+1} (-1)^{m+1-k} \stirlingone{m+1}{k} n^k. 
\end{align*}
Since
\[
(x)_s := x(x-1) \dots (x-s+1) = \sum_{k=0}^s (-1)^{s-k} \stirlingone{s}{k} x^k
\]
we obtain
\[
s_{m,n}^{(1)} - s_{m,n}^{(2)} = (-1)^m (n-1)! (n)_{m+1} = (-1)^m\binom{n-1}{m} m!\; n!.
\qedhere
\]
\end{proof}
So, in particular jeu de taquin yields uniform distribution on the double-tailed diamond $D_{m,n}$ if and only if $m \geq n$. 

Let us close this section by noting that the if-direction can also be obtained
by a simple inductive argument, which can be extended to general posets: If we
play jeu de taquin on $D_{m,m}$ with all permutations where $\pi_1$ has a fixed
value and stop the sorting procedure before $\pi_1$ is moved, then we obtain a
(non-uniform) distribution $(\alpha,\beta)$ with $\alpha+\beta=(2m-1)!$ and
$\alpha$, $\beta$ independent from $\pi_1$. As previously observed $\pi_1$ does
not change the order between $T_{1,m}$ and $T_{2,m-1}$ if and only if $\pi_1$ is
a $\RLk{m}$, \ie $\pi_1 \in \{1,2,\ldots,m\}$. After completing jeu de taquin
by moving $\pi_1$, we therefore obtain the distribution $(\alpha,\beta)$ if $\pi_1
\in\{1,2,\ldots,m\}$ and the distribution $(\beta,\alpha)$ if $\pi_1 \in
\{m+1,m+2,\ldots,2m\}$. In total each of the two standard fillings occurs $m(\alpha+\beta)$ times, \ie jeu de taquin yields a uniform
distribution on $D_{m,m}$. In the same way we can now fix $\pi_1$ in jeu de
taquin on $D_{m,n}$ with $m > n$. Inductively we obtain a uniform distribution
on $D_{m-1,n}$ for each fixed $\pi_1 \in \{1,2,\ldots,m+n\}$. Since each
fixed $\pi_1$ either always or never changes the order of the entries in the
incomparable boxes, we also obtain a uniform distribution on $D_{m,n}$.

Given an $n$-element poset $P$ and an order $\sigma$ such that jeu de taquin
yields uniform distribution, we can extend this property to the poset $P'$
obtained by adding a maximum element $m$ to $P$: First it is clear that the
total number of (dual) linear extensions remains the same, \ie $f^P = f^{P'}$.
As order for jeu de taquin choose $\sigma'|_P:=\sigma|_P$ and $\sigma'(m):=n+1$. 
Now consider all labelings $\pi$ of $P'$ where $\pi_m=i$ is fixed. If we play
jeu de taquin with all such labelings and stop before moving $i$ we obtain
(restricted to $P$) a uniform distribution among the $f^P$ different dual linear
extensions (with entries $[n+1]\setminus\{i\}$). Note that in each dual
linear extension of $P'$ the label $i$ has a unique reverse path back to the
top. Thus, moving $i$ in the last step of jeu de taquin preserves the uniform
distribution. Having a uniform distribution for each $\pi_m=i\in[n+1]$ implies uniform distribution in total.

Since jeu de taquin with row-wise order on Young tableaux yields a uniform
distribution \cite{NPSHookBijection}, it follows from the previous
observation that the poset obtained from removing the top row of $P_{k,\la}$ has
the same property. It remains an
open problem to understand why the uniform distribution is also preserved when adding the top row.

\section{A combinatorial proof of Theorem \ref{th:mainTheorem}}
\label{sec:BijectiveProofOfMT}

In this section we give a bijective proof of Theorem \ref{th:mainTheorem}. For
this purpose we define the type $\tau$ for each permutation $\pi \in \Sn_{m+n}$
by setting $\tau(\pi):=1$ if jeu de taquin with input permutation $\pi$ yields
the output tableau $T_1$, and $\tau(\pi)=-1$ if the output
tableau is $T_2$ (see Figure
\ref{fig:dmn-poset-dual-linear-extension} and Figure
\ref{fig:d65-initial-filling}).
Given two subsets $S_1,S_2 \subseteq \Sn_{m+n}$ we say that $f:S_1 \to S_2$ is
type-inverting if $\tau(\pi)=-\tau(f(\pi))$ for all $\pi \in S_1$. To give a
combinatorial proof of Theorem \ref{th:mainTheorem} we define a type-inverting involution
$\Phi_{m,n}:\Sn_{m+n}\to\Sn_{m+n}$ for all $m \geq n$. In the case $m<n$ we identify a set $\mathcal{E}$ of $\binom{n-1}{m} m!\; n!$ exceptional permutations in $\mathcal{S}_{m+n}$ of the same type. On the remaining set $\Sn_{m+n} \setminus \mathcal{E}$ we then define a type-inverting involution $\Phi_{m,n}$.

As in Section \ref{sec:JDTonDmn} let $x_i(\pi)$ and $y_i(\pi)$ denote the entries
$T_{1,m}$ and $T_{2,m-1}$ after $i$ rounds of jeu de taquin. Before giving the formal definition of $\Phi_{m,n}$ let us first state the basic ideas:

First, if the last $i$ entries of $\pi$ are in the same relative order as the last $i$ entries of $\pi'$, then $x_i(\pi) < y_i(\pi)$ if and only if $x_i(\pi') < y_i(\pi')$. This means that whether or not $x_i(\pi) < y_i(\pi)$ for $i=n+1,\ldots,n+m$ depends on the relative order of $\pi_{m+n+1-i}$, $\pi_{m+n+2-i}$, $\ldots$, $\pi_{m+n}$ but not the absolute values.

Second, we have noted in Section \ref{sec:JDTonDmn} that $\pi_i$ does not change the order
between $T_{1,m}$ and $T_{2,m-1}$ if and only if $\pi_i$ is among the $m+1-i$ smallest elements of $\{\pi_i,\pi_{i+1},\ldots,\pi_{m+n}\}$. This implies that whether or not $\pi_i$ changes the order only depends on the set of elements $\{\pi_{i+1},\pi_{i+2},\ldots,\pi_{m+n}\}$, but not their relative order. In particular, $\pi_1$ changes the order between $T_{1,m}$ and $T_{2,m-1}$ if and only if $\pi_1 > m$.

In the case $m=n$ we therefore construct an involution $\Phi_{n,n}$ on
$\Sn_{2n}$ such that the relative order of all entries in $\pi$ and the relative order of all entries in
$\pi':=\Phi_{n,n}(\pi)$ is the same if we exclude $\pi_1$ and $\pi_1'=2n+1-\pi_1$. In the case $m>n$ we let
$\Phi_{m,n}$ fix the first $m-n$ entries of each permutation and apply the
type-inverting involution $\Phi_{n,n}$ to the bottom $2n$ entries. If $m<n$ we
apply $\Phi_{m,m}$ to the smallest $2m$ entries if $\pi_1 \leq 2m$.
Else, we apply $\Phi_{m-1,m-1}$ to the smallest $2(m-1)$ entries if $\pi_2 \leq
2(m-1)$, and so on. Either one of the first $m$ entries is small enough to
apply the type-inverting involution $\Phi_{m+1-i,m+1-i}$ or $\pi_1$, $\pi_2$, $\ldots$, $\pi_m$ are all too large. In the latter case we call the permutation exceptional and exclude it from the involution $\Phi_{m,n}$. As it turns out there are exactly $\binom{n-1}{m}\;m!\;n!$ exceptional permutations all having the same type, thus proving Theorem \ref{th:mainTheorem}.

Let us now formally define the involution $\Phi_{m,n}$ in all three cases, prove the correctness and give examples. 

\subsection{Case $m=n$}
For $n\in\mathbb{N}$ and $1 \leq t \leq 2n$ define the permutation $\chi_{n,t}\in\Sn_{2n}$ by
\begin{equation}
1 \leq t \leq n: \quad\quad\quad
\chi_{n,t}(i):=
\begin{cases}
   2n+1-t & \text{if } i=t, \\
   i-1       & \text{if } t < i \leq 2n+1-t, \\
   i       & \text{otherwise.} \\
\end{cases}
\end{equation}
\begin{equation}
n+1 \leq t \leq 2n: \quad
\chi_{n,t}(i):=
\begin{cases}
   2n+1-t & \text{if } i=t, \\
   i+1       & \text{if } 2n+1-t \leq i < t, \\
   i       & \text{otherwise.} \\
\end{cases}
\end{equation}
For all $1 \leq t \leq 2n$ we have 
\[
\chi_{n,t} \circ \chi_{n,2n+1-t} = \chi_{n,2n+1-t} \circ \chi_{n,t} = \id
\]
and $\chi_{n,t}\evalat{[2n]\setminus t}$ is order-preserving. The desired involution $\Phi_{n,n}:\Sn_{2n} \to \Sn_{2n}$ is
\begin{equation}
\Phi_{n,n}(\pi) := \chi_{n,\pi_1} \circ \pi.
\end{equation}
An example can be seen in Figure \ref{fig:phi44-involution}. 
\begin{figure}[ht]
\begin{center}
\includegraphics[height=2.5cm]{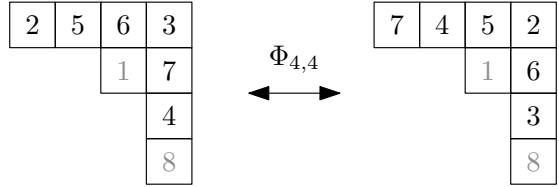}
\caption{The involution $\Phi_{4,4}$ applied to $\pi=25631748$.\label{fig:phi44-involution}}
\end{center}
\end{figure}
As composition of permutations it is clear that $\Phi_{n,n}(\pi) \in \Sn_{2n}$, and from $\chi_{n,\pi_1}(\pi_1)=2n+1-\pi_1$ it follows that
\[
\Phi_{n,n}^2(\pi) = \Phi_{n,n}(\chi_{n,\pi_1} \circ \pi) = \chi_{n,2n+1-\pi_1} \circ \chi_{n,\pi_1} \circ \pi = \pi, 
\]
\ie $\Phi_{n,n}$ is an involution. Since $\chi_{n,\pi_1}$ is order-preserving except for $\pi_1$ the entries of $\pi$ and $\pi':=\Phi_{n,n}(\pi)$ have the same relative order except for $\pi_1$ and $\pi'_1$. Therefore
\[
x_{2n-1}(\pi)<y_{2n-1}(\pi) \quad \iff \quad x_{2n-1}(\pi')<y_{2n-1}(\pi').
\]
As $\pi'_1 = 2n+1-\pi_1$ exactly one of $\pi_1 > n$ or $\pi'_1 > n$ holds, and
thus the two permutations $\pi$ and $\pi'$ are of different type, \ie
$\Phi_{n,n}$ is a type-inverting involution on $\mathcal{S}_{2n}$.

\subsection{Case $m > n$}
Given two subsets $A,B \subseteq \mathbb{N}$ with $\abs{A}=\abs{B}$, let $\sigma_{A,B}:A
\to B$ denote the unique order-preserving bijection between $A$ and $B$, \ie
the bijection satisfying $(a_1 < a_2) \rightarrow (\sigma_{A,B}(a_1) <
\sigma_{A,B}(a_2))$ for all $a_1,a_2\in A$. Obviously, we have $\sigma_{B,A} \circ
\sigma_{A,B} = \sigma_{A,B} \circ \sigma_{B,A} = \id$.

If $m > n$ and $\pi\in\Sn_{m+n}$ set
$A:=A^\pi:=\{\pi_{m-n+1},\pi_{m-n+2},\ldots,\pi_{m+n}\}$, $B:=\{1,2,\ldots,2n\}$
and $t:=t^{\pi}:=\sigma_{A,B}(\pi_{m-n+1})$. The type-inverting involution
$\Phi_{m,n}$ in this case is
\begin{equation}
\Phi_{m,n}(\pi):=
\begin{cases}
i \mapsto \pi_i & \text{if } 1\leq i \leq m-n, \\
i \mapsto \sigma_{B,A} \circ \chi_{n,t} \circ \sigma_{A,B}(\pi_i) & \text{if } m-n+1 \leq i \leq m+n.
\end{cases}
\end{equation}
Note that $\Phi_{m,n}(\pi)$ is well-defined and an element of $\Sn_{m+n}$ (see Figure \ref{fig:phi53-involution} for an example).
\begin{figure}[ht]
\begin{center}
\includegraphics[height=2cm]{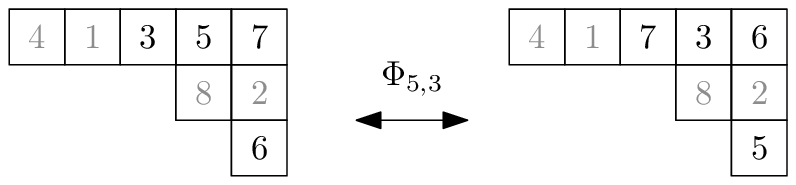}
\caption{The involution $\Phi_{5,3}$ applied to $\pi=41357826$ with $A^{\pi}=\{2,3,5,6,7,8\}$.\label{fig:phi53-involution}}
\end{center}
\end{figure}
Moreover we have 
\[
\begin{array}{ll}
\Phi_{m,n}(\pi) \evalat{\{1,\ldots,m-n\}} &= \pi, \\
\Phi_{m,n}(\pi) \evalat{\{m-n+1,\ldots,m+n\}} &= \sigma_{B,A} \circ \chi_{n,t} \circ \sigma_{A,B} \circ \pi.
\end{array}
\]
Therefore $A^{\Phi_{m,n}(\pi)}=A^\pi$ and $t^{\Phi_{m,n}(\pi)} = \sigma_{A,B}(\sigma_{B,A} \circ \chi_{n,t^{\pi}} \circ \sigma_{A,B} (\pi_{m-n+1}))= \chi_{n,t^{\pi}}(t^{\pi}) = 2n+1-t^{\pi}$. It follows that 
\[
\Phi_{m,n}^2(\pi)\evalat{\{1,\ldots,m-n\}} = \pi\evalat{\{1,\ldots,m-n\}}
\]
and
\begin{multline*}
\Phi_{m,n}^2(\pi)\evalat{\{m-n+1,\ldots,m+n\}} = \Phi_{m,n}(\sigma_{B,A} \circ \chi_{n,t^{\pi}} \circ \sigma_{A,B} \circ \pi)\evalat{\{m-n+1,\ldots,m+n\}} \\
= \sigma_{B,A} \circ \chi_{n,2n+1-t^{\pi}} \circ \sigma_{A,B} \circ \sigma_{B,A} \circ \chi_{n,t^\pi} \circ \sigma_{A,B} \circ \pi \evalat{\{m-n+1,\ldots,m+n\}} = \pi \evalat{\{m-n+1,\ldots,m+n\}},
\end{multline*}
\ie $\Phi_{m,n}^2 = \id$. As in the case $m=n$ the relative order of the last $2n-1$ entries of $\pi$ and $\pi':=\Phi_{m,n}(\pi)$ is the same. The entry $\pi_{m-n+1}$ is among the $n$ smallest elements of $\{\pi_{m-n+1},\ldots,\pi_{m+n}\}$ if and only if $\pi'_{m-n+1}$ is not among the $n$ smallest elements of $\{\pi'_{m-n+1},\ldots,\pi'_{m+n}\}$. Therefore 
\[
x_{2n}(\pi) < y_{2n}(\pi) \quad\iff\quad x_{2n}(\pi') > y_{2n}(\pi').
\]
Since $\pi_i = \pi'_i$ for all $i=1,\ldots,m-n$, the permutations $\pi$ and $\pi'$ are of different type.

\subsection{Case $m < n$}
In the case $m < n$ let us define a subset $\mathcal{E} \subseteq \Sn_{m+n}$ of exceptional permutations which we exclude from the involution: We say that $\pi=\pi_1 \dots \pi_{m+n}$ is exceptional if and only if $\pi_i > 2(m+1-i)$ for all $i=1,\ldots,m$. Note that the number of exceptional permutations is $(n-m)(n-m+1)\dots(n-1)\;n! = \binom{n-1}{m}\;m!\;n!$. Given $\pi \in \Sn_{m+n} \setminus \mathcal{E}$, let $k:=k^\pi \geq 1$ minimal such that $\pi_k \leq 2(m+1-k)$. Define
\begin{equation}
\Phi_{m,n}(\pi):=
\begin{cases}
i \mapsto \pi_i & \text{if } \pi_i > 2(m+1-k), \\
i \mapsto \chi_{m+1-k,\pi_k} (\pi_i) & \text{otherwise.}
\end{cases}
\end{equation}
An example can be seen in Figure \ref{fig:phi57-involution}.
\begin{figure}[ht]
\begin{center}
\includegraphics[height=4.5cm]{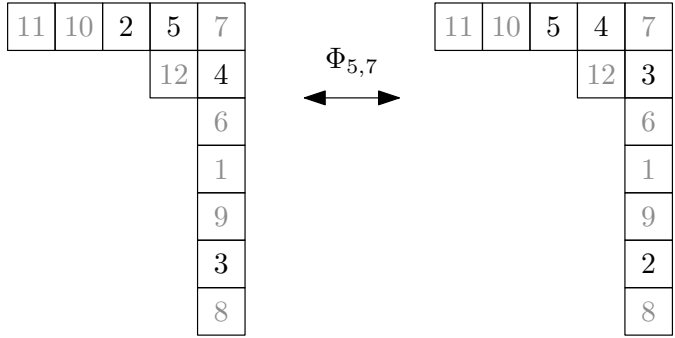}
\caption{The involution $\Phi_{5,7}$ with $k^{\pi}=3$.\label{fig:phi57-involution}}
\end{center}
\end{figure}
Note that $\Phi_{m,n}(\pi)$ is well-defined and since $\chi_{m+1-k,\pi_k} (\pi_k) = 2(m+1-k)+1-\pi_k$ we have $k^{\Phi_{m,n}(\pi)}=k^\pi$ and $\Phi_{m,n}(\pi) \in \Sn_{m+n} \setminus \mathcal{E}$. With $L:=\{1 \leq i \leq m+n : \pi_i > 2(m+1-k)\}$ it follows that
\[
\Phi_{m,n}^2(\pi) \evalat{L} = \Phi_{m,n}(\pi) \evalat{L} = \pi \evalat{L} 
\]
and
\begin{multline*}
\Phi_{m,n}^2(\pi) \evalat{[m+n] \setminus L} = \Phi_{m,n}(\chi_{m+1-k,\pi_k} \circ \pi) \evalat{[m+n] \setminus L} \\
= \chi_{m+1-k,2(m+1-k)+1-\pi_k} \circ \chi_{m+1-k,\pi_k} \circ \pi \evalat{[m+n] \setminus L} = \pi \evalat{[m+n] \setminus L},
\end{multline*}
\ie $\Phi_{m,n}$ is an involution. The relative order of the entries in
$\pi$ and $\pi':=\Phi_{m,n}(\pi)$ is the same except for $\pi_k$ and $\pi'_k$.
Since $\pi_k$ is among the $m+1-k$ smallest elements of
$\{\pi_k,\ldots,\pi_{m+n}\}$ if and only if  $\pi'_k$ is not among the $m+1-k$
smallest elements of $\{\pi'_k,\ldots,\pi'_{m+n}\}$ the involution $\Phi_{m,n}$
is type-inverting.
We can conclude the proof by noting that all exceptional permutations are of the
same type, since $\pi_i > 2(m+1-i)$ for all $i=1,\ldots,m$ implies that $\pi_i$
is not among the $m+1-i$ smallest elements of
$\{\pi_i,\pi_{i+1},\ldots,\pi_{m+n}\}$.

\section{Proof of Corollary \ref{cor:mainCorollary} \& Examples}
\label{sec:CorollaryExpectedValue}

The statement of Corollary \ref{cor:mainCorollary} can be observed by computing the number $f^{k,\la}$ of standard fillings of $P_{k,\la}$ in two different ways. On the one hand we can use the hook-length
formula \eqref{eq:HLFInsets} for insets. On the other hand we could refine the
counting \wrt the left-most entry in the second row: In each standard filling
of $P_{k,\la}$ the $k-1$ left-most entries in the first row are
$(T_{1,1},T_{1,2},\ldots,T_{1,k-1})=(1,2,\ldots,k-1)$. For $i=0,1,\ldots,\la_1$
let $f^{k,\la}_i$ denote the number of standard fillings of $P_{k,\la}$ where
the left-most entry in the second row is $T_{2,k-1}=k+i$ (see Figure
\ref{fig:insets-refined-counting}).
\begin{figure}[ht]
\begin{center}
\includegraphics[height=4.5cm]{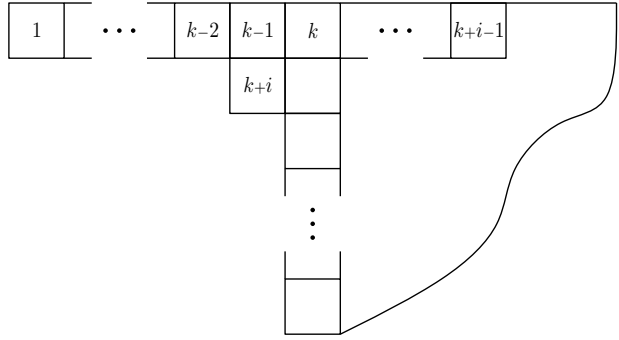}
\caption{Standard fillings counted by $f^{k,\la}_i$.\label{fig:insets-refined-counting}}
\end{center}
\end{figure}
Now note that for each fixed $i \in \{0,1,\ldots,\la_1\}$ the standard fillings
counted by $f^{k,\la}_i$ are in one-to-one correspondence with standard Young
tableaux of shape $\la$ where the left-most entry in the second row is at least
$i+1$ (by considering the entries in $\la$ and the order-preserving map).
Together with $f^{k,\la} = \sum_{i=0}^{\la_1} f^{k,\la}_i$,
\eqref{eq:FRTHookFormula} and \eqref{eq:HLFInsets} we obtain
\[
\mathbb{E}X^{\la} = \sum_{i=1}^{\la_1+1} \mathbb{P}\{X^{\la} \geq i\} = \sum_{i=0}^{\la_1} \frac{f^{k,\la}_i}{f^{\la}} = \frac{f^{k,\la}}{f^{\la}} = \prod_{i=1}^{k} \frac{n+i}{n+i-\la_{i}}.
\]
Let us apply this result to the three families of partitions in Figure \ref{fig:expectation-examples}.
\begin{example}
Consider the partition $\la = (k,1^{k-1}) \vdash 2k-1$. From Corollary \ref{cor:mainCorollary} we obtain
\[
\mathbb{E}X^\la = \prod_{i=1}^{k} \frac{2k-1+i}{2k-1+i-\la_{i}} = \frac{2k}{k}\prod_{i=2}^{k} \frac{2k-1+i}{2k-2+i}= 3-\frac{1}{k}.
\]
Of course, this could also be obtained by the elementary observation that $f^{\la}=\binom{2k-2}{k-1}$ and
\[
\mathbb{E}X^\la = \sum_{i \geq 1} \mathbb{P}(X^{\la} \geq i) = \frac{1}{\binom{2k-2}{k-1}} \left[\binom{2k-2}{k-1} + \sum_{i \geq 2} \binom{2k-i}{k-1} \right] = 1+\frac{\binom{2k-1}{k}}{\binom{2k-2}{k-1}} = 3-\frac{1}{k}.
\]
\end{example}
\begin{example}
Fix $c \geq 1$ and consider the partition $\la=(c,\ldots,c)\vdash k c$. For $k
\geq c$ it follows that
\[
\mathbb{E}X^\la = \prod_{i=1}^k \frac{kc+i}{kc+i-c} = \prod_{i=1}^{c}
\frac{k(c+1)+1-i}{kc+1-i} \xrightarrow{k \rightarrow \infty} \left(1 + \frac{1}{c} \right)^c.
\]
\end{example}
\begin{example}
Let $\la =(k,k-1,\ldots,1)\vdash \binom{k+1}{2}$ be of staircase shape. After a
short computation one obtains
\[
\mathbb{E}X^\la = \frac{\left(\binom{k+1}{2}+k\right)!!\left(\binom{k+1}{2}-k-1\right)!!}{\binom{k+1}{2}!},
\]
where $!!$ denotes the double factorial, \ie $(2n)!! = 2^n n!$ and $(2n-1)!! = \frac{(2n)!}{2^n n!}$. 
By applying Stirling's formula one can show that asymptotically $\mathbb{E}X^\la \sim e \approx 2.71828$. 
\end{example}

\bibliography{bib131107}
\bibliographystyle{alpha}

\end{document}